\documentclass[12pt, reqno]{amsart}

\usepackage{verbatim}
\usepackage{amssymb, amsmath, amsthm}
\usepackage{hyperref}
\usepackage[alphabetic,backrefs,lite]{amsrefs} 
\usepackage{amscd}   
\usepackage{fullpage}
\usepackage[all]{xy} 
\usepackage{pdflscape}
\usepackage{setspace}
\DeclareFontEncoding{OT2}{}{} 

\usepackage[usenames, dvipsnames]{color}

\usepackage{graphicx}
\usepackage{tikz}
\usetikzlibrary{matrix,arrows,decorations.pathmorphing}

\usepackage{tikz-cd}
\usepackage{mathtools}
\usepackage{enumitem}

\newtheorem{lemma}{Lemma}[section]
\newtheorem{theorem}[lemma]{Theorem}

\newtheorem{prop}[lemma]{Proposition}
\newtheorem{cor}[lemma]{Corollary}

\newtheorem{claim*}{Claim}

\newtheorem{defn}[lemma]{Definition}

\newtheorem*{theorem*}{Theorem}

\newtheorem{construction}{Construction}

\theoremstyle{definition}
\newtheorem{remark}[lemma]{Remark}

\newcommand{\A}{{\mathbb A}}
\newcommand{\G}{{\mathbb G}}

\newcommand{\PP}{{\mathbb P}}
\newcommand{\C}{{\mathbb C}}
\newcommand{\F}{{\mathbb F}}
\newcommand{\Q}{{\mathbb Q}}
\newcommand{\R}{{\mathbb R}}
\newcommand{\Z}{{\mathbb Z}}

\newcommand{\Xbar}{{\overline{X}}}

\newcommand{\kbar}{{\overline{k}}}
\newcommand{\Kbar}{{\overline{K}}}
\newcommand{\Ksep}{{K^{\operatorname{sep}}}}
\newcommand{\ksep}{{k^{\operatorname{sep}}}}
\newcommand{\Fbar}{{\overline{\F}}}

\newcommand{\Ybar}{{\overline{Y}}}

\newcommand{\kk}{{\mathbf k}}

\newcommand{\calC}{{\mathcal C}}

\newcommand{\calF}{{\mathcal F}}

\newcommand{\calH}{{\mathcal H}}
\newcommand{\calI}{{\mathcal I}}

\newcommand{\calL}{{\mathcal L}}

\newcommand{\calO}{{\mathcal O}}

\newcommand{\calX}{{\mathcal X}}
\newcommand{\calY}{{\mathcal Y}}

\newcommand{\frakD}{{\mathfrak D}}

\newcommand{\frakS}{{\mathfrak S}}

\usepackage[mathscr]{euscript}

\newcommand{\scrM}{{\mathscr M}}

\newcommand{\scrO}{{\mathscr O}}

\newcommand{\scrX}{{\mathscr X}}

\newcommand{\hideqed}{\renewcommand{\qed}{}} 

\DeclareMathOperator{\HH}{H}

\DeclareMathOperator{\rk}{rk}

\DeclareMathOperator{\inv}{inv}

\DeclareMathOperator{\im}{im}

\DeclareMathOperator{\Gal}{Gal}

\DeclareMathOperator{\Br}{Br}

\DeclareMathOperator{\Pic}{Pic}

\DeclareMathOperator{\Spec}{Spec}

\DeclareMathOperator{\Proj}{Proj}

\DeclareMathOperator{\ev}{ev}

\DeclareMathOperator{\tors}{tors}
\DeclareMathOperator{\et}{\textrm{\normalfont \'et}}

\DeclareMathOperator{\rank}{rank}

\DeclareMathOperator{\NS}{NS}
\DeclareMathOperator{\Bl}{Bl}

\DeclareMathOperator{\Hilb}{Hilb}
\DeclareMathOperator{\coarse}{coarse}
\DeclareMathOperator{\LambdaK}{\Lambda_{\textup{K3}}}

\newcommand{\isom}{\cong}

\newcommand{\into}{\hookrightarrow}

\numberwithin{equation}{section}
\numberwithin{table}{section}

\newcommand{\defi}[1]{\textsf{#1}} 

\newcommand{\Sbar}{\overline{S}}

\title{Odd order obstructions to the Hasse principle on general K3 surfaces}

\author{Jennifer Berg}
\address{Department of Mathematics, Rice University MS 136, Houston, TX
77251-1892, USA}
\email{jb93@math.rice.edu}
\urladdr{http://math.rice.edu/\~{}jb93}

\author{Anthony V\'arilly-Alvarado}
\address{Department of Mathematics, Rice University MS 136, Houston, TX
77251-1892, USA}
\email{av15@math.rice.edu}
\urladdr{http://math.rice.edu/\~{}av15}

\date{}

\begin{document}

	\begin{abstract}
		We show that odd order transcendental elements of the Brauer group of a K3 surface can obstruct the Hasse principle. We exhibit a general K3 surface $Y$ of degree 2 over $\Q$ together with a three torsion Brauer class $\alpha$ that is unramified at all primes except for 3, but ramifies at all 3-adic points of $Y$. Motivated by Hodge theory, the pair $(Y, \alpha)$ is constructed from a cubic fourfold $X$ of discriminant 18
		birational to a fibration into sextic del Pezzo surfaces over the projective plane. Notably, our construction does not rely on the presence of a central simple algebra representative for $\alpha$. Instead, we prove that a sufficient condition for such a Brauer class to obstruct the Hasse principle is insolubility of the fourfold $X$ (and hence the fibers) over $\Q_3$ and local solubility at all other primes.
	\end{abstract}

	\maketitle

\section{Introduction}\label{sec:Introduction}
	


The purpose of this paper is to give an arithmetic application of the connection between certain Brauer classes of order $3$ on polarized K3 surfaces of degree 2 and special cubic fourfolds of discriminant 18. We use this connection to construct K3 surfaces over number fields failing to satisfy the Hasse principle.  Although the Hodge theoretic correspondence between special cubic fourfolds and K3 surfaces was systematically studied by Hassett nearly 20 years ago~\cite{Has99}, the results herein rely on the timely work of Addington et.\ al. \cite{AHTVA}, which proved that a general cubic fourfold $X$ of discriminant $18$ contains a sextic elliptic ruled surface. Projecting away from the surface yields a fibration into sextic del Pezzo surfaces over $\PP^2$, out of which one can recover a degree 2 K3 surface $Y$ together with a subgroup $B$ of $\Br(Y)[3]$. Even then, making transcendental Brauer elements explicit is seldom a simple task. Exhibiting a central simple algebra representing an element of $B$ remains elusive without deeper knowledge of sextic del Pezzo surfaces over non-closed fields. Instead, our key insight is to make repeated use of the Lang-Nishimura lemma to translate local invariant computations for Brauer classes on the K3 surface into local solubility information about the total space of the fibration and hence the cubic fourfold $X$ itself. 

\subsection*{Odd order obstructions}
A smooth projective geometrically integral variety $X$ over a number field $k$ is a counterexample to the \defi{Hasse principle} if simultaneously the set $X(\A_k)$ of adelic points of $X$ is nonempty while the set of $k$-points $X(k)$ is empty.  Fix an algebraic closure $\kbar$ of $k$, and let $\Xbar := X \times_k \kbar$. The Brauer group $\Br X := H^2_{\et}(X, \G_m)_{\tors}$ of $X$ contains two important subgroups $\Br_0(X) \subseteq \Br_1(X) \subseteq \Br(X)$ that organize the study of Brauer--Manin obstructions to the Hasse principle (see~\S\ref{sec:InvariantComputations} for details):
	\begin{align*}
		\text{the \defi{constant classes }} \Br_0(X) &:= \im \left(\Br(k) \to \Br(X) \right), \text{and}\\
		\text{the \defi{algebraic classes }} \Br_1(X) &:= \ker \left(\Br(X) \to \Br(\Xbar) \right).
	\end{align*}
Elements of $\Br(X)$ that are not algebraic are called \defi{transcendental}. 

Brauer classes of odd order, whether algebraic or transcendental, had previously only been known to obstruct weak approximation on K3 surfaces \cites{Preu, IS15}. Ieronymou and Skorobogatov~\cite{IS15} showed that  elements of odd order can never obstruct the Hasse principle on smooth diagonal quartics in $\PP^3_\Q$.  Their work prompted them to ask if there exists a locally soluble  K3 surface over a number field with an odd-order Brauer--Manin obstruction to the Hasse principle~\cite{IS15}*{p.~183}. Since then, Skorobogatov and Zarhin~\cite{SZ16} have shown that such classes cannot obstruct the Hasse principle on Kummer surfaces. Recently, Corn and Nakahara~\cite{CN17} gave a positive answer to Ieronymou and Skorobogatov's question, by showing that a $3$-torsion \emph{algebraic} class can obstruct the Hasse principle on a degree 2 K3 surface over $\Q$. 
Our main result shows that $3$-torsion \emph{transcendental} Brauer classes can obstruct the Hasse principle on K3 surfaces.

\begin{theorem} \label{thm: MainThm}
	There exists a K3 surface $Y$ over $\Q$ of degree 2, together with a class $\alpha$ in $\Br Y[3]$ such that
	\[  Y(\A_\Q) \ne \emptyset \qquad \text{and} \qquad Y(\A_\Q)^{\{\alpha\}} = \emptyset. \]
	Moreover, $\Pic \Ybar \isom \Z$, and hence $\Br_1(Y)/\Br_0(Y) = 0$. In particular, there is no algebraic Brauer-Manin obstruction to the Hasse principle on $Y$.
\end{theorem}

In contrast to the manifold constructions of explicit transcendental Brauer classes on K3 surfaces developed over the last 15 years~\cites{Wittenberg,Ieronymou,HVAV11,EJ13,Preu,HVA13, IS15, MSTVA}, we do not write down an explicit central simple algebra representative for $\alpha$ over the function field $\kk(Y)$.  It would be worthwhile to do so; this would likely require a conceptual approach different from the one we take in this paper.

\subsection*{Cubic fourfolds of discriminant 18} 
The pair $\left(Y, \langle\alpha\rangle\right)$ is associated to a smooth cubic fourfold $X \subset \PP^5$ of discriminant 18. Work of Huybrechts~\cite{Huybrechts} and McKinnie et.\ al.~\cite{MSTVA} shows that every subgroup $\langle\alpha\rangle$ of order 3 in the Brauer group of a polarized complex K3 surface $Y$ of degree 2 with $\Pic(Y) \isom \Z \ell$ must arise from one of three types of auxiliary varieties, one of which is a special cubic fourfold of discriminant 18 (see Theorem~\ref{thm:MSTVA}).  This lattice-theoretic computation is realized geometrically in~\cite{AHTVA} and~\cite{Kuz17}, where it is shown that a general cubic fourfold $X$ of discriminant $18$ contains a sextic elliptic surface $T$, and projecting away from this surface gives a fibration $\pi\colon X' := \Bl_T(X) \to \PP^2$ into del Pezzo surfaces of degree $6$. Crucially, the structure morphism $\Hilb_{3n+1}(X'/\PP^2) \to \PP^2$ of the Hilbert scheme parametrizing twisted cubics in the fibers of $\pi$ has a Stein factorization
\[
\Hilb_{3n+1}(X'/\PP^2) \to Y \to \PP^2,
\]
consisting of an \'etale $\PP^2$-bundle followed by a finite degree $2$ map branched over a smooth sextic curve.  Hence $Y$ is a K3 surface, and $\alpha \in \Br(Y)[3]$ is the class corresponding to the Severi-Brauer surface bundle $\Hilb_{3n+1}(X'/\PP^2) \to Y$. The covering involution $\iota\colon Y \to Y$ induces a map $\iota^*\Br(Y) \to \Br(Y)$ such that $\iota^*(\alpha) = \alpha^{-1}$. In this sense, the scheme $\Hilb_{3n+1}(X'/\PP^2)$ is seen to recover the subgroup $\langle\alpha\rangle \subset \Br(Y)[3]$.

\subsection*{Highlights} 

In proving Theorem ~\ref{thm: MainThm}, we use several ideas developed in~\cites{HVAV11,HVA13,MSTVA}.  Nevertheless, odd order transcendental obstructions to the Hasse principle mediated by cubic fourfolds of discriminant 18 present significant theoretical and computational challenges not previously encountered:

\begin{enumerate}[leftmargin=*]
\item We are not aware of a computationally feasible method to write down equations for the scheme $\Hilb_{3n+1}(X'/\PP^2)$ that are suitable for our arithmetic application. Our approach instead uses ideas of Corn~\cite{Corn}: the generic fiber $S$ of the morphism $\pi$ is a del Pezzo surface of degree $6$ over the function field $\kk(\PP^2)$.  Its base extension $S_{\kk(Y)}$ contains two triples of pairwise-skew $(-1)$-curves, each defined over $\kk(Y)$.  Contracting each triple gives a del Pezzo surface of degree $9$ over $\kk(Y)$, i.e., we get two Severi-Brauer surfaces over $\kk(Y)$. Corn shows that the corresponding $3$-torsion Brauer classes in $\Br \kk(Y)$ are inverses of each other.  We use work of Addington et.\ al., Koll\'ar, and Kuznetsov~\cites{AHTVA,Kollar,Kuz17} to show that the subgroup of $\Br \kk(Y)$ thus obtained coincides with the image of the subgroup $\langle\alpha\rangle$ under the inclusion $\Br(Y)\hookrightarrow \Br \kk(Y)$. Hence, the classes obtained via Corn's construction are unramified over $Y$.  This description of the $3$-torsion Brauer classes is more amenable to explicit computations.

\medskip

\item A classical theorem of Wedderburn shows that the nontrivial Brauer classes in $\langle\alpha\rangle$ can be represented over $\kk(Y)$ by cyclic algebras; quite generally, any central simple algebra of degree 3 over a field is cyclic~\cite{KMRT}*{\S 19}.  We do not give such a representation here; even producing a suitable degree 3 cyclic extension of $\kk(Y)$ is nontrivial.  Instead we use a repeated application of the Lang-Nishimura lemma to show that local invariants at a finite place $v$ are non-trivial for every $P_v \in Y(k_v)$ as long as the cubic fourfold $X$ has no $k_v$-points. In fact, the most delicate step in our construction is writing down a cubic fourfold $X$ \emph{of discriminant 18} such that $X(\Q_3)$ is empty. We accomplish this by a taking a cleverly chosen linear combination of cubics each depending only on $3$ variables, that are $(\Z/3\Z)$-insoluble when considered as plane cubics.  Note that a cubic in at least 4 variables always has $(\Z/3\Z)$-points by the Chevalley-Warning theorem, and $\approx 99.259 \%$ of plane cubic curves over $\Q_3$ have $\Q_3$-rational points (\cite{BCF16}*{Theorem 1}).  These considerations compelled us to cast a wide net to find the fourfold $X$.

\medskip

\item Finding the primes of bad reduction of the K3 surface $Y$ is also computationally challenging; it requires the factorization of a 366 digit integer! We use an idea in~\cite{HVA13} to do this: we compute in turn an integer whose prime factors encode the primes of bad reduction for the cubic fourfold $X$, with the expectation that these numbers share a large common factor (which can be quickly computed using the Euclidean algorithm).  Notably, to compute this additional integer, we use the grevlex monomial ordering for Gr\"obner basis computations, despite the fact that grevlex is not an order suitable for general elimination theory (see \S\ref{sec:ComputationalTricks} to see why this idea works).

\medskip

\item We use work of Kuznetsov~\cite{Kuz17} to show $\alpha$ can only ramify at places where $Y$ has bad reduction, whence these are the only places where the local invariants of $\alpha$ can be nontrivial. 

\medskip

\item We impose the condition that the K3 surface $Y$ has good reduction at $p = 2$ to ensure that $\alpha$ does not ramify there. This requires producing, when possible, a model for $Y$ which reduces to an equation of the form $w^2 + w g_1 + g_2 = 0$ over $\F_{2}$. The Jacobian criterion can then be used to check for smoothness.

\medskip

\item To recover the discriminant locus $f(x,y,z) = 0$ of the del Pezzo fibration $\pi$, we were impelled to use an interpolation approach as the natural Gr\"obner basis calculation is computationally infeasible. We sample integral coprime pairs $(y_0, z_0)$ of small height, each of which gives a degree 12 univariate polynomial which is the specialization of $f(x,y_0,z_0)$ of $f(x,y,z)$. By viewing $F$ as a polynomial with indeterminate coefficients, each specialization gives linear relations among the coefficients. With a sufficient number of specializations, one can reconstruct $F$ using Gaussian elimination.
\end{enumerate}

\medskip

\noindent In addition to the contributions above, we mention two standard difficulties that arise in this type of construction.

\medskip

\begin{enumerate}[leftmargin=*]\setcounter{enumi}{6}
\item There is no a priori guarantee that the K3 surface afforded by the construction above
has geometric Picard rank one. We use work of
van Luijk~\cite{vL07}, and Elsenhans and Jahnel~\cites{EJ08,EJ11} on K3 surfaces of degree 2 to produce such a K3 surface. 

\medskip

\item To compute the local invariants of a class $\alpha$ at finite places of bad reduction $p \neq 3$, we use~\cite{HVA13}, which shows that if the singular locus of the reduction of $X$ at $p$ consists of at most $7$ ordinary double points, then the local invariant map $X(\Q_p) \to \frac{1}{p}\Z/\Z$ has constant image. It thus suffices to compute its value at one $\Q_p$-point of $X$. 
\end{enumerate}	

\section*{Acknowledgments}

We thank Nicolas Addington, Asher Auel, Brendan Hassett, Michael Stoll, and Olivier Wittenberg for illuminating conversations and comments.  A.\ V.-A.\ was partially supported by NSF grant DMS-1352291. All computations presented here were done with {\tt Magma} \cite{MAGMA}.

\section{Intimations from Hodge Theory}\label{sec:StartingData}

The purpose of this section is to briefly indicate, along the lines of~\cites{vG05,HVAV11,HVA13,MSTVA}, why one would consider cubic fourfolds of discriminant 18 as a source for transcendental $3$-torsion Brauer classes on general degree $2$ K3 surfaces. We refer the reader to~\cite{MSTVA} for a full account of these ideas. All varieties in this section are complex.

\subsection{Lattices and Hodge Theory}

Let $Y$ be a complex projective K3 surface of degree $2d$, i.e., a smooth projective integral surface over $\C$ such that $\HH^1(Y,\scrO_Y) = 0$ and $\omega_Y \isom \scrO_Y$, together with an ample line bundle with self-intersection $2d$. The singular cohomology group $\HH^2(Y,\Z)$ is a lattice with respect to the intersection form, and by \cite{LP81}, we can write
\[
	\HH^2(Y,\Z) \isom U^3 \oplus E_8(-1)^2 =: \LambdaK, 
\]
where $U$ is the hyperbolic plane (i.e. an even unimodular lattice whose base extension to $\R$ has signature (1,1)), and $E_8(-1)$ is the unique negative definite even unimodular lattice of rank eight. In particular, $\LambdaK$ is even, unimodular, and has signature (3,19).

Write $\NS(Y)$ for the Neron-Severi group of $Y$; we call $T_Y := \NS(Y)^\perp \subset H^2(Y,\Z)$ the \defi{transcendental lattice} of $Y$. By \cite{vG05}*{\S2}, an element of order $n$ in $\Br(Y)$ gives rise to a surjective homomorphism $\alpha\colon T_Y \to \Z/n\Z$, whose kernel is a sublattice of index $n$. Conversely, a sublattice $\Gamma \subset T_Y$ of index $n$ determines a subgroup $\langle \alpha \rangle \subset \Br(Y)$ of order $n$.

Generalizing results of van Geemen~\cite{vG05}, McKinnie et.\ al. classify sublattices of $T_Y$, up to isomorphism, when $\NS(Y) \isom \Z \ell$ with $\ell^2 = 2d > 0$, and $n = p$ is a rational prime~\cite{MSTVA}. We summarize their findings in the case where $d = 1$ and $p = 3$.

\begin{theorem}
\label{thm:MSTVA}
Let $Y$ be a complex projective K3 surface with $\NS(Y) \cong \Z \ell$, $\ell^2 = 2$, and let $\langle \alpha \rangle \subset \Br(Y)$ be a subgroup of order $3$. Then either

\medskip

\begin{enumerate}

\item there is a unique primitive embedding $\Gamma_{\langle \alpha \rangle} \hookrightarrow \Lambda_{K3}$, in which case  surjectivity of the period map gives rise to a degree $18$ K3 surface $X$ associated to the pair $(Y,\langle\alpha\rangle)$; or

\medskip

\item $\Gamma_{\langle \alpha \rangle}(-1) \cong \langle h^2, T\rangle^\perp \subseteq {\rm H}^4(X,\Z)$, where $X \subset \PP^5$ is a cubic fourfold of discriminant 18, $h$ is the hyperplane class, and $T$ is a surface in $X$ not homologous to $h^2$; or,

\medskip

\item $\Gamma_{\langle \alpha \rangle}$ is a lattice with discriminant group $\Z/2\Z \times (\Z/3\Z)^2$.
\end{enumerate} 
\end{theorem}

\begin{proof}
	This is a special case of~\cite{MSTVA}*{Theorem~9}, together with~\cite{MSTVA}*{Proposition~10} and the material in~\cite{MSTVA}*{\S2.6}.
\end{proof}

Theorem~\ref{thm:MSTVA} suggests that a cubic fourfold $X$ of discriminant $18$ should encode a pair $\left(Y,\langle\alpha\rangle\right)$, where $Y$ is K3 surface of degree $2$, together with a subgroup $\langle\alpha\rangle\subset \Br(Y)[3]$.  In  the next few sections, we explain a construction that gives a geometric manifestation of this lattice-theoretic calculation.

\section{Special Cubic fourfolds of discriminant 18}\label{sec:Cubics18}

Recall that a \defi{special cubic fourfold} $X \subset \PP^5$ is a smooth cubic hypersurface that contains a surface $T$  not homologous to a complete intersection~\cite{Has99}. Write $h$ for the restriction of the hyperplane class of $\PP^5$ to $X$, and assume that the lattice $K := \langle h^2, T \rangle \subset H^4(X,\Z)$ is saturated. The \defi{discriminant} of $(X,K)$ is the determinant of the Gram matrix of $K$. Special cubic fourfolds of discriminant $D \equiv 0, 2 \bmod 6$ form an irreducible divisor $\calC_D$ in the coarse moduli space of cubic fourfolds.  In particular, the coarse moduli space $\calC_{18}$ of special cubic fourfolds of discriminant $18$ is a $19$-dimensional quasi-projective variety.

We recall results from \S2 of~\cite{AHTVA}, which show that a general element of $\calC_{18}$ is birational to a fibration into sextic del Pezzo surfaces over $\PP^2$.

\begin{theorem}{\cite{AHTVA}} \label{thm: AHTVA}
	Let $X\subset \PP^5$ be a general cubic in $\calC_{18}$.
	\medskip

 	\begin{enumerate}
		\item The fourfold $X$ contains an elliptic ruled surface $T$ of degree 6 with the property that the linear system of quadrics in $\PP^5$ containing $T$ is two-dimensional, and its base locus is a complete intersection $\Pi_1 \cup T \cup \Pi_2$, where $\Pi_1$ and $\Pi_2$ are disjoint planes.
		\medskip

		\item The generic fiber of the map $\pi \colon X' := \Bl_T(X) \to \PP^2$ induced by the linear system of quadrics containing $T$ is a del Pezzo surface of degree 6 over $\C(\PP^2)$. 
		\medskip

		\item  The morphism $\pi\colon X' \to \PP^2$ is a \defi{good del Pezzo fibration} (see \cite{AHTVA}*{Definition 11}) if and only if the discriminant locus is a curve of degree $12$ with two irreducible sextic components $B_I$ and $B_{II}$ that intersect transversely; moreover, $B_I$ is a smooth plane sextic and $B_{II}$ has 9 cusps.
	\end{enumerate}
\end{theorem}

As the curve $B_I$ is a smooth sextic in $\PP^{2}$, the double cover of $\PP^2$ branched along this sextic is a smooth K3 surface $Y$ of degree 2.  When working over nonclosed base fields, this double cover can have quadratic twists.  In~\S\ref{sec:HPCounterexample} we explain how to use Galois-theoretic data of the configuration of $(-1)$-curves on the fibers of $\pi$ to chose the correct quadratic twist $Y$. 

\section{Du Val families of sextic del Pezzo surfaces}\label{sec:duVal}

Theorem~\ref{thm: AHTVA} shows that a general element $X$ of $\calC_{18}$ yields a K3 surface $Y$ of degree $2$ by way of a family of sextic del Pezzo surfaces.  In this section, we exploit recent work of Kuznetsov~\cite{Kuz17}, which generalizes the notion of a good del Pezzo fibration to that of a du Val family, as well as some of the results and ideas in~\cite{AHTVA}, to show how a subgroup of order $3$ in $\Br(Y)$ arises from such a family.  The results of this section are used in \S\ref{ss:goodred} to show that over a number field, the elements of this subgroup are unramified at places of good reduction for $Y$.

Throughout this section, $k$ denotes a field.

\begin{defn}
A \defi{sextic du Val del Pezzo surface} is a normal integral projective surface over $k$ with at worst du Val singularities, with ample anticanonical class $-K_X$ such that $K_X^2 = 6$.
\end{defn}

\begin{defn}
A \defi{du Val family of sextic del Pezzos} is a flat projective morphism $f\colon\scrX \to S$ of $k$-schemes such that, for each geometric point $s \in S$, the fiber $\scrX_s$ is a sextic du Val del Pezzo surface.
\end{defn}

\begin{theorem} \label{thm: Kuznetsov}
Let $\calO$ be a discrete valuation ring with fraction field $k$ and residue field $\F$. Let $\pi\colon \scrX' \to \PP^2_{\calO}$ be a flat projective morphism such that $\pi_k\colon \scrX'_k \to \PP^2_k$ and $\pi_\F\colon \scrX'_\F \to \PP^2_\F$ are du Val families of sextic del Pezzos. Let $\calH = \Hilb_{3n+1}(\scrX'/\PP^2_{\calO})$ be the relative Hilbert scheme parametrizing twisted cubics in the fibers of $\pi$. Then the Stein factorization
\[
\calH \to \calY \to \PP^2_{\calO}
\]
consists of an \'etale-locally trivial $\PP^2$-bundle over $\calY$ followed by a degree $2$ cover of $\PP^2_{\calO}$. 
\end{theorem}

\begin{proof}
Let $f\colon\scrX \to S$ be a du Val family of sextic del Pezzos, and let $\scrM(\scrX/S)$ be the relative stack of semistable sheaves on the fibers of $f$ with Hilbert polynomial $(3n + 1)(n+1)$. Kuznetsov shows in ~\cite{Kuz17}*{Proposition 5.3} that $\scrM(\scrX/S)$ is a $\G_m$-gerbe on the course moduli space $\calY := (\scrM(\scrX/S))_{\coarse}$ with an obstruction given by a class $\beta \in \Br(\calY)$ of order $3$.  Moreover, letting $\calH = \Hilb_{3n+1}(\scrX/S)$ be the relative Hilbert scheme parametrizing twisted cubics in the fibers of $f$, Kuznetsov shows in~\cite{Kuz17}*{Proposition 5.14} that $\calH$ is flat over $S$, and that there is a morphism of functors $\calH \to \scrM(\scrX/S)$ such that the composition $\calH \to \scrM(\scrX/S) \to \calY$ is an \'etale-locally trivial $\PP^2$-bundle.  The bundle $\calH \to \calY$ of Severi-Brauer varieties is associated with the Brauer class $\beta$~\cite{Kuz17}*{Proposition 5.14}. Finally,~\cite{Kuz17}*{Theorem 5.2} shows that $\calY \to S$ is a finite flat morphism of degree $2$.  Taken together, these facts imply that $\calH \to \calY \to S$ is the Stein factorization of the structure morphism $\calH \to S$.  The Theorem follows from an application of the above ideas to the families $\pi_F$ and $\pi_k$, taking $S$ to be $\PP^2_\F$ and $\PP^2_k$, respectively.  
\end{proof}

Being an \'etale $\PP^2$-bundle, the morphism $\calH_k \to \calY_k$ over the generic point of $\Spec \calO$ in Theorem~\ref{thm: Kuznetsov} gives rise to a class $\alpha \in \Br(\calY_k)[3]$.  The covering involution $\iota\colon \calY_k\to \calY_k$ induces a map on Brauer groups such that $\iota^*(\alpha) = \alpha^{-1}$~\cite{Kuz17}*{Proposition~5.3}. Thus, the scheme $\calH_k$ recovers the subgroup $\langle\alpha\rangle \subset \Br(Y)[3]$.

\section{Brauer classes associated to sextic del Pezzo Surfaces}\label{sec:Geometry} 

The proof of Theorem ~\ref{thm: Kuznetsov} demonstrates that the construction of the relative Hilbert scheme $\calH \to \PP^2$ parametrizing twisted cubics in the fibers of $\pi \colon X' \to \PP^2$ is valid over a field, and thus yields a Brauer class on our K3 surface $Y$. However, without knowledge of equations for $\calH$, this representation of the Brauer class is impractical for computational purposes. In lieu of this, we use work of Corn~\cite{Corn} to provide an alternative construction of a Brauer class  which arises from the generic fiber of the del Pezzo fibration $\pi$. A priori, this class is contained only in $\Br \kk(Y)[3]$. To prove that it is in fact unramified on $Y$ and moreover equal to the class corresponding to $\calH$, we employ work of Koll\'ar \cite{Kollar} on the geometry of Severi-Brauer varieties; this is the content of Proposition~\ref{prop: EquivAlgebras}. In order to describe these constructions, we must first recall some properties of del Pezzo surfaces of degree six over nonclosed fields.

\subsection{Sextic del Pezzo surfaces}
\label{ss:dP6s}

Let $S$ be a del Pezzo surface of degree six over a field $k$. Over a fixed algebraic closure $\kbar$ of $k$, the surface $\Sbar$ is isomorphic to the blowup of the projective plane $\PP^2_{\kbar}$ at 3 non-collinear points, $\{P_1, P_2, P_3\}$. The ($-1$)-curves on $\Sbar$ consist of the exceptional divisors and the proper transforms of lines joining pairs of the points. More precisely, letting $E_i$ denote the (class in $\Pic(\Sbar)$ of) the exceptional curve corresponding to $P_i$, and $L$ the class of the pullback of a line in $\PP^2_{\kbar}$ not passing through any of the $P_i$, the six exceptional curves in $\Sbar$ are
\[
E_1, E_2, E_3, L - E_1 - E_2, L - E_2 - E_3, L - E_1 - E_3.
\]
The group $\Pic \Sbar$ is isomorphic to $\Z^4$, with a basis given by  $\{ L, E_1, E_2, E_3 \}$, and intersection pairing is determined by the relations $L^2 =1$, $E_i^2 = -\delta_{ij}$, and $E_i.L = 0$.  Hence, the six exceptional curves intersect in a hexagonal configuration (Figure~1).

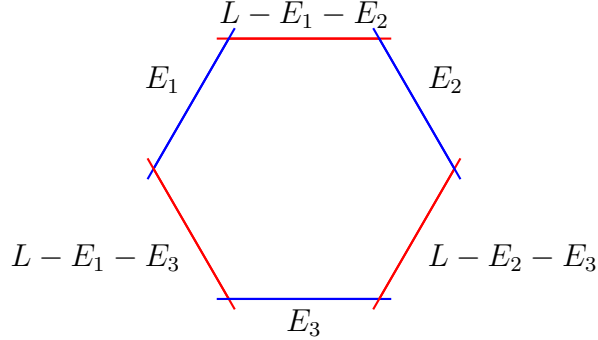
\begin{figure}[h] \label{fig: S2action}
\begin{tikzpicture}[scale = .8]
   \newdimen\Rasd;
	\Rasd=2.5cm;
   \draw (0:2.5cm);

           \foreach \x/\l in
   { 60/$E_2$,
   120/$L- E_1 - E_2 $,
   180/$E_1$,
   240/$L - E_1 - E_3$,
   300/$E_3 $,
   360/$L- E_2 - E_3 $
  }
  \draw[postaction={decorate}] ({\x-60}:\Rasd) -- node[auto,swap]{\l} (\x:\Rasd);
  \draw[red, thick] (60:2.5cm)++(.2,0) -- (120:2.5cm);
  \draw[red, thick] (120:2.5cm)++(180:.2) -- (60:2.5cm);
  \draw[blue, thick] (120:2.5cm)++(60:.2) -- (180:2.5cm);
  \draw[blue, thick] (180:2.5cm)++(240:.2) -- (120:2.5cm);
  \draw[red, thick] (180:2.5cm)++(120:.2) -- (240:2.5cm);
  \draw[red, thick] (240:2.5cm)++(300:.2) -- (180:2.5cm);
  \draw[blue, thick] (240:2.5cm)++(180:.2) -- (300:2.5cm);
  \draw[blue, thick] (300:2.5cm)++(360:.2) -- (240:2.5cm);
  \draw[red, thick] (300:2.5cm)++(240:.2) -- (0:2.5cm);
  \draw[red, thick] (360:2.5cm)++(60:.2) -- (300:2.5cm);
  \draw[blue, thick] (0:2.5cm)++(300:.2) -- (60:2.5cm);
  \draw[blue, thick] (60:2.5cm)++(120:.2) -- (0:2.5cm);
\end{tikzpicture}
\caption{The hexagon of exceptional curves on a del Pezzo surface of degree $6$}
\end{figure}

The action of the absolute Galois group on $\Pic \Sbar$ preserves intersections, and thus induces an automorphism on the hexagon of lines, which has dihedral automorphism group $\frakD_{12} = \frakS_2 \times \frakS_3$.  We obtain a homomorphism $\Gal(\kbar/k) \to\frakS_2 \times \frakS_3$. Projection onto either factor yields cocycles with values in $\frakS_2$ or $\frakS_3$, and thereby determines a pair $(K,M)$, where $K$ and $M$ are \'etale quadratic and cubic algebras over $k$, respectively. As observed in \cite{CTKM07}, the algebras $K$ and $M$ are naturally isomorphic to those associated with the Galois action on the lines of the del Pezzo surface $\Sbar$. That is, $K$ can be identified with the extension over which the lines split into two Galois stable skew triples
\begin{equation}
\label{eq:GaloisInvariantSets}
\{E_1,E_2,E_3\}\qquad\text{and}\qquad\{L-E_1-E_2,L - E_2 - E_3, L - E_1 - E_3\}.
\end{equation}
Thus, over $K$ it is possible to blow down these two Galois stable skew triples, and obtain a del Pezzo surface of degree $9$.

\begin{prop} \label{prop: CornDelPezzoData}
	Let $S$ be a del Pezzo surface of degree 6 over a field $k$. There exists a field $K$ such that $[K:k] =1$ or 2, and surfaces $X_1$ and $X_2$ defined over $K$ such that:
	\begin{enumerate}
		\item there exist morphisms $\pi_{X_i}: S_K \to X_i$ which exhibit $S_K$ as the blowup of $X_i$ at a $G_K$ stable set of three non-collinear points $\{P_{i1}, P_{i2}, P_{i3}\}$,
		\item the preimages of these points form a full set of exceptional curves in $\Sbar$,
		\item and $X_1$ and $X_2$ are Severi-Brauer varieties of dimension 2.
	\end{enumerate}
	Moreover, the Severi-Brauer varieties $X_1, X_2$ give rise to opposite algebras in $\Br K$.
\end{prop}

\begin{proof} 
	Apply \cite{Corn}*{Proposition 2.1} and \cite{Corn}*{Theorem 2.3}.  Although \cite{Corn}*{Proposition 2.1} as written requires the field $k$ to have characteristic zero, del Pezzo surfaces are separably split, i.e., $\Pic(\Sbar) = \Pic(S\times_k \ksep)$~\cite{LeidenLectures}*{\S1.4}, so there is no need to require $k$ to have characteristic zero.
\end{proof}

\begin{remark} \label{rmk: CornLinSystems}
We explain how to make the maps $\pi_{X_1}$ and $\pi_{X_2}$ of Proposition~\ref{prop: CornDelPezzoData} explicit; they are obtained by respectively contracting the Galois-invariant sets of pairwise-skew $(-1)$-curves in~\eqref{eq:GaloisInvariantSets}. The divisor classes $3L$ and $3(2L - E_1 - E_2 - E_3)$ are defined over $K$, as the following decompositions illustrate:
\begin{align*}
3L &= -K_{S} + E_1 + E_2 + E_3, \\
3(2L - E_1 - E_2 - E_3) &= -K_{S} + (L - E_1 - E_2) + (L - E_2 - E_3) + (L - E_1 - E_3).
\end{align*}
A Riemann-Roch calculation shows that the linear systems corresponding to these divisor classes are both $10$-dimensional.
The resulting $K$-morphisms $\phi_{|3L|}\colon S_K \to \PP^9_K$ and $\phi_{|3(2L - E_1 - E_2 - E_3)|}\colon S_K \to \PP^9_K$ contract precisely the sets of $(-1)$-curves in~\eqref{eq:GaloisInvariantSets}. Geometrically, the map $\phi_{|3L|}$ is the composition of the blowdown $\phi_{|L|}\colon \Sbar\to \PP^2_\Kbar$ followed by the $3$-uple embedding $\PP^2_\Kbar \to \PP^9_\Kbar$, so the image of $\phi_{|3L|}$ is a Brauer-Severi surface in $\PP^9_K$; this is the surface $X_1$, and $\pi_{X_1} = \phi_{|3L|}$; mutatis mutandis, the image of $\pi_{X_2} = \phi_{|3(2L - E_1 - E_2 - E_3)|}$ is the Brauer-Severi surface $X_2$.
\end{remark}

\begin{remark}
The field $K$ in Proposition~\ref{prop: CornDelPezzoData} coincides with the \'etale algebra $K$ of~\S\ref{ss:dP6s} whenever the latter is a field.
\end{remark}

\begin{cor} \label{cor: BrauerSubgroup} 
Let $\pi \colon X' \to \PP^2$ be a good del Pezzo fibration over a field. Let $S$ be the generic fiber, a smooth sextic del Pezzo surface over the function field $k: = \kk(\PP^2)$. Let $K$ be the quadratic \'etale algebra associated to $S$ as in~\S\ref{ss:dP6s}.  Assume that $K$ is a field. Then the two Severi-Brauer varieties $X_1$ and $X_2$ from Proposition~\ref{prop: CornDelPezzoData} each give rise to the same subgroup in $\Br K[3]$.
\qed
\end{cor}

\subsection{Key Observation} \label{subsec: KeyObservation}

With notation as in Corollary~\ref{cor: BrauerSubgroup}, let $Y$ denote the normalization of $\PP^{2}$ in $K$. Then $Y$ is a double cover of $\PP^2$, and the analysis in \cite[Proposition~18]{AHTVA} shows that it is branched along a smooth plane sextic, and is thus a K3 surface of degree 2. On the other hand, we saw in the proof of Theorem~\ref{thm: Kuznetsov} that to the data of $\pi$ we may associate the Hilbert scheme $\calH_k = \Hilb_{3n+1}(X'/\PP^2)$, whose structure morphism has a Stein factorization
\[
\calH_k \to \calY_k \to \PP^2_k
\]
consisting of an \'etale $\PP^2$-bundle followed by a double cover.  We saw that the map $\calH_k \to \calY_k$ in fact gives rise to a subgroup $\Gamma \subseteq \Br(\calY_k)[3]$. The remainder of this section is devoted to a proof of the following proposition.

\begin{prop} \label{prop: EquivAlgebras}
	The surfaces $\calY_k$ and $Y$ coincide, and the image of $\Gamma \subseteq \Br(\calY_k)[3]$ for the inclusion $\Br(Y) \hookrightarrow \Br \kk(Y)$ coincides with the subgroup of $\Br \kk(Y)[3]$ from Corollary~\ref{cor: BrauerSubgroup}.  In particular, the Brauer classes in Corollary~\ref{cor: BrauerSubgroup} are unramified on the K3 surface $Y$.
\end{prop}

\subsection{Twisted line bundles} \label{subsec: TwistedLineBundles}

The following definition, due to Koll\'ar~\cite{Kollar}, provides the bridge necessary to connect the two groups in Proposition~\ref{prop: EquivAlgebras}.

\begin{defn} Let $S$ be a geometrically reduced, geometrically connected, proper $K$-scheme. Given a line bundle $\calL$ on $\Sbar$, let $T(\calL)$ denote the category of vector bundles $\calF$ on $S$ such that $\calF_{\Kbar} \isom \oplus_i \calL$, a direct sum of copies of $\calL$. Then $\calL$ is called a \defi{twisted line bundle} on $S$ if $T(\calL)$ contains a non-zero vector bundle.
\end{defn}

\begin{lemma}[\cite{Kollar}*{Lemma~8}] Let $\calL$ be a twisted line bundle on $S$. Then there is a unique vector bundle $E(\calL) \in T(\calL)$ such that every other member of $T(\calL)$ is a sum of copies of $E(\calL)$. 
\qed
\end{lemma}

Let $e := \rank E(\calL)$. Then $L^{(e)} := \det E(\calL)$ is a line bundle on $S$ such that 
\begin{equation}
\label{eq:e}
 L^{(e)}_{\Kbar} \isom \calL^e.
 \end{equation}

\begin{lemma}[\cite{Kollar}*{Claim 7.3}] \label{lem: KollarGaloisFixed}
	Let $\calL$ be a line bundle on $\Sbar$. Then $\calL$ is a twisted line bundle on $S$ if and only if
	$\calL^\sigma \isom \calL$ for every $\sigma \in \Gal(\Ksep/K)$.
\qed
\end{lemma}

If $\calL$ is a line bundle on $S$ then $|\calL| = \PP(H^0(S,\calL)^{\vee})$ is the linear system associated to $\calL$. However, if $\calL$ is a twisted line bundle on $S$, there is no natural way to define $H^0(S,\calL)$. Instead, Koll\'ar provides several equivalent notions of a twisted linear system $|\calL|$ associated to $\calL$, each one of which yields \emph{the same} Severi-Brauer variety (up to $K$-isomorpshim)~\cite{Kollar}. 

\begin{construction}
 \label{construction1} 
 Let $\calL$ be a twisted line bundle on $S$ with $e := \rank E(\calL)$. Then there is a line bundle $L^{(e)}$ on $S$ such that $\calL_{\Kbar}^{(e)} \isom \calL^{e}$ (see~\eqref{eq:e}) and hence, in the notation of \cite{Kollar} there exists an embedding
\[ v_{e, \calL} \colon |\calL_{\Kbar}| \into |\calL^e| \isom |L^{(e)}|_{\Kbar}.\]
\end{construction}

\begin{lemma}[\cite{Kollar}*{Corollary 37.4}]  The image of $v_{e, \calL}$ is a Severi-Brauer $K$-variety, denoted by $|\calL|$.\qed
\end{lemma}

\begin{construction} \label{construction2} Let $\calL$ be a twisted line bundle on $S$. Let $|\calL|$ denote the irreducible component of the Hilbert scheme of $S$ parametrizing subschemes $H \subset S$ such that $H_{\Kbar}$ is in the linear system $|\calL_{\Kbar}|$. By Lemma~\ref{lem: KollarGaloisFixed}, since $\calL$ is a twisted line bundle we have $\calL^\sigma \isom \calL$ for each $\sigma \in \Gal(\Ksep/K)$. Hence $|\calL|$ is invariant under the Galois group, and thus by \cite{Kollar}*{Lemma 22}, it is a $K$-variety. Moreover, since $|\calL|_{\Kbar} \isom |\calL_{\Kbar}| \isom \PP(H^0(S,\calL_{\Kbar})^{\vee}) $, it is in fact a Severi-Brauer variety. 
\end{construction}

\subsubsection{An extended example}
\label{ss: twistedlinebundlesondP6}
	Let $\pi \colon X' \to \PP^2$ be a du Val family of sextic del Pezzos over a field.  The generic fiber $S$ of $\pi$ is a smooth sextic del Pezzo surface over the function field $k: = \kk(\PP^2)$. Let $K$ be the quadratic \'etale algebra associated to $S$ as in~\S\ref{ss:dP6s}.  Assume that $K$ is a field. Consider the line bundles on $S_{\Kbar}$ given by $\calL_1 := \calO(L)$ and $\calL_2:= \calO(2L - E_1 - E_2 - E_3)$. We claim that $\calL_1, \calL_2$ are twisted line bundles on $S_{K}$.  To see this, let $\sigma \in \Gal(\Kbar/K)$. Recall that $K$ is precisely the extension over which the triples of pairwise skew $(-1)-$curves in~\eqref{eq:GaloisInvariantSets} are Galois stable. The anticanonical divisor on $S$ is $-K_S \sim 3 L - E_1 - E_2 - E_3$, thus in $\Pic(S_{K})$ we have the relations:
	\begin{eqnarray*}
	3L & \sim & -K_S + (E_1 + E_2  + E_3) \\
	3(2L - E_1 - E_2 - E_3) & \sim & -K_S + (L- E_1 - E_2) + (L - E_2 - E_3) + (L - E_1 - E_3)
	\end{eqnarray*}
	hence $(\calL_1)^\sigma \isom \calL_1$, and $\calL_2^\sigma \isom \calL_2$.

\begin{lemma} \label{lem: rankOfE}
	Let $e_1 := \rank E(\calL_1)$ and $e_2 := \rank E(\calL_2)$. Assume that $\calL_1$ and $\calL_2$ are not line bundles on $S_K$. Then $e_1 = e_2 = 3$.
\end{lemma}

\begin{proof}
	Consider the rank 3 vector bundle $\calI_1$ on $S_{\Kbar}$ defined by 
	\begin{align*}
	\calI_1 & := \calO( (L- E_1 - E_2) + (E_1) + (E_2)) + \calO( (L- E_2 - E_3) + (E_2) +(E_3))\\
	& \,\, + \calO( (L- E_1 - E_3) + (E_1) + (E_3) ).
	\end{align*}
	Then $\calI_1 \isom \calO(L)^{\oplus 3}$, and as $\{E_1, E_2, E_3 \}$ is a Galois invariant set, $\calI_1$ descends to a rank 3 vector bundle on $S_K$. Thus, $\calI_1 \in T(\calL_1)$, and $e_1 \le 3$.

	By assumption, $\calL_1$ is not a line bundle on $S_K$, hence $e_1 \ne 1$. Finally, if $e_1 = 2$, then $3L - 2L = L$ would be defined over $K$, a contradiction. The proof that $e_2 = 3$ follows similarly by considering the rank 3 vector bundle $\calI_2 \isom \calO(2L - E_1 - E_2 - E_3)^{\oplus 3}$ over $\Kbar$ defined by
	\begin{align*}
	\calI_2 & :=  \calO( (L - E_1 - E_3) + (L - E_1 -E_2) + (E_1) ) \\
	& \,\, + \calO( (L - E_1 - E_2) + (L - E_2 -E_3) + (E_2) )   \\
	& \,\, + \calO( (L - E_2 - E_3) + (L - E_1 -E_3) + (E_3) ). \tag*{\qed}
	\end{align*}
	\hideqed
\end{proof}

\begin{lemma} \label{lem: KollarToCorn} 
	The Severi-Brauer varieties $|\calL_1|, |\calL_2|$  are isomorphic to the Severi-Brauer varieties $X_1$, $X_2$ associated to $S_K$ as in Proposition \ref{prop: CornDelPezzoData}.
\end{lemma}

\begin{proof}
	Consider the twisted line bundle $\calL_1$. By Lemma~\ref{lem: rankOfE}, we have $e_1 = 3$, so Construction~\ref{construction1} applied to $\calL_1$ geometrically gives the composition of the morphism induced by the linear system $|({\calL_1})_\kbar| = |L|$ with the $3$-uple embedding of $\PP^{2}=\PP(H^0(S_\Kbar,\calO(L))^\vee)$ to $\PP^9$, which is exactly the morphism $\phi_{|3L|} = \pi_{X_1}$ of Remark~\ref{rmk: CornLinSystems}. A similar analysis for $\calL_{2}$ recovers the morphism $\pi_{X_2}$. 
\end{proof}

\subsection{Proof of Proposition~\ref{prop: EquivAlgebras}}
Recall $S$ is the generic fiber of the good del Pezzo fibration $\pi \colon X' \to \PP^2$; it is a sextic del Pezzo surface over $\kk(\PP^2)$. Let $\calL_1$, $\calL_2$ be the twisted line bundles of \S\ref{ss: twistedlinebundlesondP6}. By Lemma~\ref{lem: KollarToCorn}, Construction~\ref{construction1} applied to $\calL_1$ and $\calL_2$ on  $S_K$ yields the nontrivial elements of $\Br \kk(Y)[3]$ from Corollary~\ref{cor: BrauerSubgroup}. On the other hand, Construction~\ref{construction2} applied to $\calL_1$ and $\calL_2$ on $S_K$ yields the nontrivial elements of the image of $\Gamma \subseteq\Br(\calY_k)[3]$ for the inclusion $\Br(Y) \hookrightarrow \Br \kk(Y)$. By~\cite{Kollar}, these two constructions give the same Brauer classes. \qed

\section{Local Invariant Computations}\label{sec:InvariantComputations}

Let $Y$ be a smooth projective geometrically integral variety over a number field $k$. For $B \subseteq \Br(Y)$, let
\[ 
	Y(\A_k)^B := \bigg\{ (P_v) \in Y(\A_k) : \sum_v \inv_v \alpha(P_v) = 0 \textup{ for all } \alpha \in B \bigg\},
\]
where $\inv_v\colon \Br k_v \to \Q/\Z$ is the invariant map from local class field theory. The quantity $\inv_v \alpha(P_v)$ can be nonzero only a finite number of places of $k$: the archimedean places, the places of bad reduction of $X$, and the places where the class $\alpha$ is ramified.  Global class field theory shows the set $Y(\A_k)^B$ satisfies
\[
Y(k) \subseteq Y(\A_k)^B\subseteq Y(\A_k).
\]
We say that $Y$ has a \defi{Brauer--Manin obstruction} to the Hasse principle if $Y(\A_k) \neq \emptyset$ and $Y(\A_k)^B = \emptyset$~\cite{Man71}.

When $Y$ is a K3 surface of degree $2$ and $B = \{\alpha\}$ consists of a nontrivial element in the group $\Gamma$ of Proposition~\ref{prop: EquivAlgebras}, we may use a result from \cite{HVA13} (which is an application of~\cite{CTS}) to show that local invariants are constant at certain finite places $v$ of bad reduction of $Y$, where the singular locus satisfies a technical hypothesis. After recalling the necessary result, we show that $\alpha$ can ramify only over places of bad reduction for $Y$. In the special case $k = \Q$, we give sufficient conditions for the  invariant of $\alpha$ to be nontrivial at all $v$-adic points of $Y$. Finally, we note that a point $P_\infty$ at an archimedean place has $\inv_\infty \alpha(P_\infty) = 0$ because local invariants at archimedean places have an image lying in $\frac{1}{2}\Z/\Z\subset \Q/\Z$, and $\alpha$ has order $3$.

\subsection{Places of bad reduction with mild singularities}

Let $k$ be a finite extension of $\Q_p$ with a fixed algebraic closure $\kbar$, let $\calO$ denote the ring of integers of $k$, and let $\F$ denote the residue field. 

\begin{prop}{ \cite{HVA13}*{Prop. 4.1, Lemma 4.2} } \label{prop: MildBadRed}
	Suppose that $p \ne 2$, and that $\ell \neq p$ is prime. Let $Y$ be a K3 surface over $k$, and let $\pi \colon \calY \to \Spec \calO$ be a flat proper morphism from a regular scheme with $Y = \calY \times_{\calO} k$. Assume that the singular locus of the closed fiber $\calY_0 := \calY_{\Fbar}$ consists of $r < 8$ points, each of which is an ordinary double point. If $Y(k) \ne \emptyset$, then for $\alpha \in \Br(Y)\{ \ell\}$, the image of the evaluation map $\ev_\alpha \colon Y(k) \to \Br(k)$ is constant.
	\qed
\end{prop}

\begin{cor} \label{cor: ConstInvtsMildBadRed} 
With notation as in the beginning of \S\ref{sec:InvariantComputations}, if $v$ is a place of bad reduction for the K3 surface $Y$ for which the singular locus of the reduction consists of $r < 8$ points, each of which is an ordinary double point, then $\inv_v \alpha(P_v) = 0$ for all $P_v \in Y(k_v)$.
\end{cor}
\begin{proof} Suppose that $P_v \in Y(k_v)$ and let $P_v' \in Y(k_v)$ be its image under the covering involution $\iota\colon Y \to Y$. Then,
\[ - \inv_v \alpha(P_v) = \inv_v( \alpha^{-1}) (P_v)  = \inv_v (\iota^\ast \alpha) (P_v) =  \inv_v \alpha( \iota(P_v) ) = \inv_v \alpha(P_v'), \]
where $\iota(\alpha) = \alpha^{-1}$ by the arguments of \S\ref{sec:duVal}.
Since the image of the map $\ev_{\alpha}$ is constant by Proposition ~\ref{prop: MildBadRed}, this forces $\inv_v \alpha(P_v) = 0$ for all $P_v \in Y(k_v)$.
\end{proof} 

\subsection{No ramification at places of good reduction}
\label{ss:goodred}

Let $v$ be a finite place of a number field $k$.  Denote by $k_v$ the corresponding completion, and write $\calO_v$ and $\F_v$ for the associated ring of integers and residue field, respectively. Let $\pi\colon\scrX' \to \PP^2_{\calO_v}$ be a flat projective morphism such that $\pi_{k_v}\colon \calX'_{k_v} \to \PP^2_{k_v}$ and $\pi_{\F_v}\colon \calX'_{\F_v} \to \PP^2_{\F_v}$ are sextic du Val del Pezzo families.  Assume further that in the Stein factorization
\[
\calH := \Hilb_{3n+1}(\calX'/\PP^2_{\calO_v}) \to \calY \to \PP^2_{\calO_v}
\]
the surfaces $\calY_{k_v}$ and $\calY_{\F_v}$ are smooth, and hence are K3 surfaces, by the analysis in~\cite[Proposition~18]{AHTVA}. The proof of Theorem~\ref{thm: Kuznetsov} shows that the map $\calH_{k_v} \to \calY_{k_v}$ is an \'etale $\PP^2$-bundle, and hence gives rise to a class $\alpha_{k_v} \in \Br(\calY_{k_v})$.

\begin{lemma} \label{lem: PlacesOfRamification}
	With notation as above, for each $P_v \in \calY(k_v)$, we have $\inv_v \alpha(P_v) = 0$.
\end{lemma}

\begin{proof}
	Theorem~\ref{thm: Kuznetsov} implies that the element $\alpha_{k_v} \in \Br(\calY_{k_v})$ spreads out to a class $\alpha \in \Br(\calY)$. Since $\calY(O_v) = \calY(k_v)$, it follows that $\alpha(P_v) \in \Br(\calO_v) = 0$.
\end{proof}

\subsection{Invariants at the remaining places of bad reduction} 

In this section we use the notation of Corollary~\ref{cor: BrauerSubgroup} and \S\ref{subsec: KeyObservation}, specializing to the case $k = \Q$. The following lemma gives a sufficient condition to guarantee that the local invariants of $\alpha$ at $v$-adic points of $Y$ are always non-trivial, where $v$ is a place of bad reduction not satisfying the hypotheses of Proposition ~\ref{prop: MildBadRed}.
  
\begin{lemma} \label{lem: InsolubleCubic} 
	If $X'(\Q_v) = \emptyset$, then $\inv_v \alpha(P) \ne 0$ for all $P \in Y(\Q_v)$.
\end{lemma}

\begin{proof}
	Since the evaluation map $\ev_\alpha$ is locally constant, and the invariant map $\inv_v$ is an isomorphism onto its image, it suffices to prove that the invariant takes on the desired values in a $v$-adically dense open subset of $Y(\Q_v)$. Thus, let $P \in Y(\Q_v)$, let $Q$ be its image in $\PP^2(\Q_v)$, and suppose that $Q \not \in B_I \cup B_{II}$ (see~Theorem~\ref{thm: AHTVA}(3)), so that the corresponding fiber of the good del Pezzo fibration $\pi \colon X' \to \PP^2$ is a smooth sextic del Pezzo surface, denoted $X'_{Q}$.
	
	By construction, the fiber $\calH_P$ of the map $\calH \to Y$ over $P$ is the Severi-Brauer variety $\text{SB}(\alpha(P))$, and thus $\inv_v \alpha(P)$ is nontrivial if and only if $\calH_P(\Q_v)$ is empty. The proof of Proposition~\ref{prop: EquivAlgebras} applied to the smooth del Pezzo fiber $X'_{Q}$ of $\pi$ over $\Q_3$ shows that blowing down one of the triples in~\eqref{eq:GaloisInvariantSets} on $X'_{Q}$ gives a del Pezzo surface of degree $9$ birational to the fiber $\calH_P$.  Hence, the Lang--Nishimura lemma says that $\calH_P(\Q_v)$ is empty if and only if $X'_{Q}(\Q_v)$ is empty. We conclude that a sufficient condition to guarantee $X'_{Q}(\Q_v) = \emptyset$, and hence that $\calH_{P}(\Q_v) = \emptyset$, for each $P \in Y(\Q_v)$ is to simply have $X'(\Q_v)$ be empty.
\end{proof}

\section{Counterexample to the Hasse Principle}\label{sec:HPCounterexample}
Let $\PP^5 := \Proj \Q [x_0,x_1,x_2,x_3,x_4,x_5]$, and define quadrics cut out by the zero locus of
\begin{align*}
 Q_1 :=& -x_0x_3 + x_2x_3 - x_0x_4 + x_1x_4 + 3x_2x_4 + 5x_0x_5 - x_1x_5, \\
 Q_2 :=& -x_1x_3 + 5x_0x_4 - 2x_2x_4 - 2x_0x_5 + 5x_1x_5 + x_2x_5, \\
 Q_3 :=& -2x_2x_3 - x_0x_4 - 2x_1x_4 - 2x_2x_4 + x_1x_5.
\end{align*}
Each one of these quadrics contains the planes 
\[
	\Pi_1 := \{ x_0 = x_1 = x_2 = 0 \} \hspace{1 em } \textup{ and } \hspace{1 em } \Pi_2 := \{ x_3 = x_4 = x_5 = 0 \}.
\]
We obtain a sextic elliptic ruled surface $T$ by saturating the ideal $\langle Q_1,Q_2,Q_3\rangle$ with respect to the product ideal $I(\Pi_1)I(\Pi_2)$. The surface $T$ is cut out by the vanishing of $Q_1, Q_2, Q_3$ and the two cubics
\begin{align*}
	C_1 :=&2x_3^3 + 5x_3^2x_4 + x_3x_4^2 + 14x_4^3 - 20x_3^2x_5 - 26x_3x_4x_5 - 11x_4^2x_5 + 47x_3x_5^2 + 30x_4x_5^2 + 5x_5^3, \\
    C_2 :=&2x_0^3 - x_0^2x_1 - 2x_0x_1^2 - x_1^3 + 47x_0^2x_2 + 10x_0x_1x_2 + x_1^2x_2 - 11x_0x_2^2 - 18x_1x_2^2 - 4x_2^3.
\end{align*}

Now comes the most delicate point of the construction: we pick a cubic 
\[
	C \in I(T) = \langle Q_1,Q_2,Q_3,C_1,C_2\rangle
\]
whose vanishing locus $X$ is a $3$-adically insoluble smooth cubic fourfold. Note that each of the cubic curves $\{C_1 = 0\}$ and $\{C_2 = 0\}$ depends only three variables.  They both lack $(\Z/3\Z)$-points, \emph{when considered as plane cubic curves}. Hence, for any choice of integers $(x_0,x_1,x_2) \neq (0,0,0)$ and $(x_3,x_4,x_5) \neq (0,0,0)$, we have
\[
C_1(x_3,x_4,x_5) \not\equiv 0 \bmod 3\quad\text{and}\quad C_2(x_0,x_1,x_2) \not\equiv 0 \bmod 3.
\]
As a consequence, the cubic 
\[
	\{3C_1 + C_2 = 0\}
\]
fails to have points over $\Z/9\Z$. However, it is not smooth over $\Q$. We modify this cubic by adding a suitable nonic multiple of a cubic in $\langle Q_1,Q_2,Q_3\rangle$; this way we can guarantee that the resulting fourfold $X$ is smooth, and that its associated (twisted) K3 surface enjoys a host of other properties we require, such as good reduction modulo $2$.  To wit, we take
\begin{align*} 
	C :=&\ 9\left[(x_0 + x_1 + x_2 + x_5)Q_1 + (x_1 + x_3)Q_2 + (x_0 + x_1 + x_4 + x_5)Q_3\right] + 3C_1 + C_2\\
	=&\ 2x_0^3 - x_0^2x_1 - 2x_0x_1^2 - x_1^3 + 47x_0^2x_2 + 10x_0x_1x_2 + x_1^2x_2 - 11x_0x_2^2 - 18x_1x_2^2 - 4x_2^3 \\
	&\quad + 18x_0^2x_3 + 18x_0x_1x_3 + 9x_1^2x_3 + 18x_0x_2x_3 + 18x_1x_2x_3 + 18x_2^2x_3 + 9x_1x_3^2 + 6x_3^3  \\
	&\quad + 36x_0^2x_4 + 9x_0x_1x_4 + 18x_1^2x_4 - 9x_0x_2x_4 + 18x_1x_2x_4 + 18x_2^2x_4 - 27x_0x_3x_4 \\
	&\quad + 18x_2x_3x_4 + 15x_3^2x_4 + 27x_0x_4^2 - 36x_2x_4^2 + 3x_3x_4^2 + 42x_4^3 - 90x_0^2x_5 - 72x_0x_1x_5 \\
	&\quad - 45x_1^2x_5 - 18x_1x_2x_5 + 36x_0x_3x_5 - 45x_1x_3x_5 + 9x_2x_3x_5 - 60x_3^2x_5 - 54x_0x_4x_5 \\
	&\quad + 27x_1x_4x_5- 18x_2x_4x_5 - 78x_3x_4x_5 - 33x_4^2x_5 - 90x_0x_5^2 + 141x_3x_5^2 + 90x_4x_5^2 + 15x_5^3.
\end{align*}
The above discussion shows that $X(\Z/9\Z) = \emptyset$, and hence $X(\Q_3) = \emptyset$.

The discriminant locus of the map $\pi \colon X'= \Bl_T(X) \to \PP^2 = | \calI_T(2)^\vee| $ is a reducible curve of degree 12 with two irreducible components, $B_I$ and $B_{II}$, with $B_I$ the smooth plane sextic over $\Q$ cut out by 
\begin{align*}
	f &:= 17279788x^6 + 21966980x^5y + 5209685x^4y^2 - 10091766x^3y^3 - 9449085x^2y^4 \\
	&- 3512294xy^5 - 510755y^6 + 81563000x^5z + 46799342x^4yz - 48304566x^3y^2z \\
	&- 68669390x^2y^3z - 29936552xy^4z - 4960696y^5z + 132675265x^4z^2 - 24537700x^3yz^2 \\
	&- 153420566x^2y^2z^2 - 94604246xy^3z^2 - 18001746y^4z^2 + 88262884x^3z^3 - 116707356x^2yz^3 \\
	&- 139178230xy^2z^3 - 36604266y^3z^3 + 12231034x^2z^4 - 90599148xyz^4 - 40695955y^2z^4 \\
	&- 11073000xz^5 - 22207274yz^5 - 3652475z^6,
\end{align*}
and a double cover $Y \to \PP^2$ ramified along $B_I$ given by 
\[
	\delta w^2 = f(x,y,z)
\]
for some $\delta \in \Q^\times$, is a K3 surface of degree 2.  To determine the value of $\delta$, it suffices to look at a single smooth fiber $S_{[x_0,y_0,z_0]} := \pi^{-1}([x_0,y_0,z_0])$ of $\pi\colon X' \to \PP^2$. Indeed, by Proposition \ref{prop: CornDelPezzoData}, $\delta$ is determined, up to squares, by the property that over the extension $\Q(\sqrt{f(x_0,y_0,z_0)/\delta})$ each of the skew triples of $(-1)$-curves in $S_{[x_0,y_0,z_0]}$ is defined.  This way we check that, in our case, we can take $\delta = 1$.

\subsection{Primes of bad reduction} \label{subsec: PrimesOfBadRed} The Jacobian criterion shows that, with the possible exception of $p = 2$, the primes of bad reduction of the model we have for the K3 surface $Y$ coincide with the primes of bad reduction of the plane sextic $f = 0$.  The primes of bad reduction of this plane sextic must divide the generator $m$ of the ideal obtained by saturating
\[
\left\langle f, \frac{\partial f}{\partial x}, \frac{\partial f}{\partial y}, \frac{\partial f}{\partial z} \right\rangle \subset \Z[x,y,z]
\]
by the irrelevant ideal, and computing its third elimination ideal.  A Gr\"obner basis calculation over $\Z$ shows that
\begin{align*}
	m = & \,\, 5183655491723801700167749249923539759212595989317043450941205133922477 \\
	& \,\,4983736928297221979398638986972977399830115294629898991772309007577733 \\
	& \,\,9340845846248825313938176384776583116617867319433833462119034621956835 \\
	& \,\,0164014449486497652543584687852235647461802798371746334045253660849102 \\
	& \,\,6619155248743682227444947206215266448570301668208072897742377424333883 \\
	& \,\,22718119551330227149710084257339986806689874438360.
\end{align*}
Standard factorization methods yield a few small prime power factors of $m$:
\[
	m = 2^3 \cdot 3^{10} \cdot 5 \cdot 29^2 \cdot 2851^2 \cdot 1647622003^2 \cdot m';
\]
the remaining factor $m'$ has 366 decimal digits. Factorization of such a large integer is a computationally difficult problem. However, it is reasonable to expect that the integer encoding the primes of bad reduction for the cubic fourfold $X$ has a large greatest common divisor with $m$. Indeed, by an analogous Gr\"obner basis calculation over $\Z$, we find that the primes of bad reduction of $X$ must divide 
\begin{align*}
	n := & \,2785635380281567358627612747160273008954813709377362118360363769288 \\
	&\,3266943448924557748494585237466118988173996980092848560234291977128 \\
	&\,5136691999562751842259916853009979290324458976115151429311777291929 \\
	&\,4109180259515866460532106060165319658125293838250844956494005775
\end{align*}

The greatest common divisor is computed instantaneously via the Euclidean algorithm:
\begin{align*} 
	\gcd(m,n) := &\,\, 10916958733847554724999087148257562876963755787681804271807512 \\
	&\,\,39742865771648776681654603351155791354915762094311121777611445 \\
	&\,\,310537427891011052531888686655803840674098627355.
\end{align*}
This is still a 172 digit integer! However, modern computer technology is in fact able to factor both the $\gcd(m,n)$ and $m/\gcd(m,n)$ using elliptic curve factorization.  We find that, with the possible exception of $p = 2$,  the primes of bad reduction for $Y$ are:
\begin{align*}
	& 5, 29, 2851, 1647622003, 8990396491695741359, 381640024919828593698301, \\
	&23298439293572123109021711335092905690126256356826741414312843163784586626801847, \\
	&70630573062884782978729488744706657246824151776978742375050861454515493652288934 \\
	&35340410321256513135415547594556084340880768251657255814972524891.
\end{align*}
Finally, we show that the K3 surface $Y$ has good reduction at $p = 2$, by making a suitable change of variables over $\Q$.  Reducing the sextic $f(x,y,z)$ we obtain a perfect square:
\[
(x^2y + x^2z + xy^2 + y^3 + yz^2 + z^3)^2.
\]
Applying the $\Q$-transformation $w\mapsto 2w  + \left(x^2y + x^2z + xy^2 + y^3 + yz^2 + z^3\right)$ to $Y$, we obtain a new model for $Y$, all of whose coefficients are divisible by $4$.  Dividing out by this factor we obtain a new model for $Y$, whose reduction modulo $2$ is
\begin{align*}
w^2 &+ (x^2y + x^2z + xy^2 + y^3 + yz^2 + z^3)w \\
&+ x^6 + x^5y + x^4y^2 + x^4yz + x^3yz^2 + x^3z^3 + xyz^4 + y^6 + y^4z^2 + y^3z^3 + y^2z^4 + yz^5 + z^6 = 0.
\end{align*}
An application of the Jacobian criterion shows that this surface is quasi-smooth over $\F_2$, and since the surface is a Cartier divisor in the weighted projective space $\Proj \F_2[x,y,z,w]$, where $x$, $y$, $z$, and $w$ have respective weights $1$, $1$, $1$, and $3$, it follows that the surface is smooth.  Hence $Y$ has good reduction at $p = 2$.

\subsection{Local points} \label{subsec: LocalPoints}

By the Weil conjectures, if $p > 22$ is a prime such that $Y$ has smooth reduction $Y_p$ at $p$, then $Y_p$ has a smooth $\F_p$-point that can be lifted to a smooth $\Q_p$-point, by Hensel's lemma. Thus, to show that $Y$ is everywhere locally soluble, it suffices to verify solubility of $Y$ at $\R$ (clear), at $\Q_p$ for primes $p \le 19$, and those primes $p > 19$ for which $Y$ has bad reduction. The results are recorded in \hyperref[fig1invts]{Table 1}.

\begin{table}[!htb] \small \label{fig1invts}
\caption{Verification that $Y$ has $\Q_p$ points at small $p$ and primes of bad reduction}
\begin{align*}
\begin{tabular}{|r|r|r|r|r|}
\hline
\hspace{.2 em }$p$ \hspace{.2 em} & \hspace{.2 em} $x$  \hspace{.2 em} & $y$ & $z$ & $w$ \\
\hline \hline
2 & $-1$ & 1 & 0 & 20892553 \\
\hline
3 & $-1$ & $-1$ & $-1$ & $-520729088$ \\
\hline
5 & $-1$ & 0 & $-1$ & 317286496 \\
\hline
7 & $-1$ & $-1$ & $-1$ & $-520729088$ \\
\hline
11 &  $-1$ & $-1$ & 1 & 1303120 \\
\hline
13 & $-1$ & $-1$ & $-1$ & $-520729088$ \\
\hline
17 & $-1$ & $-1$ & 1 & 1303120 \\
\hline
19 & $-1$ & $-1$ & $-1$ & $-520729088$ \\
\hline
29 & $-1$ & $-1$ & $-1$ & $-520729088$ \\
\hline
2851 & $-1$ & 0 & $-1$ & 317286496 \\
\hline
1647622003 & $-1$ & $-1$ & $-1$ & $-520729088$ \\
\hline
8990396491695741359 & $-1$ & $-1$ & $-1$ & $-520729088$ \\
\hline
381640024919828593698301 & $-1$ & $-1$ & $-1$ & $-520729088$ \\
\hline
232984392935721231090217113 & {} & {} & {} & {} \\
350929056901262563568267414  & $-1$ & $-1$ & 0 & 20892553 \\
14312843163784586626801847 & {} & {} & {} & {} \\
\hline
706305730628847829787294887 & {} & {} & {} & {} \\
447066572468241517769787423 & {} & {} & {} & {} \\
750508614545154936522889343 & $-1$ & $-1$ & $-1$ & $-520729088$ \\
534041032125651313541554759 & {} & {} & {} & {} \\
455608434088076825165725581 & {} & {} & {} & {} \\
4972524891 & {} & {} & {} & {} \\
\hline
\end{tabular} & \hspace{3 em}
\end{align*}
\end{table}

\subsection{Picard Rank 1}

In this section we show that $Y$ has geometric Picard rank 1. This will guarantee that the obstruction to the Hasse principle arising from $\alpha$ is genuinely transcendental. We follow the strategy laid out in~\cite{HVA13}*{\S5.3} (which crucially uses~\cite{EJ11}), with one significant update: since the publication of~\cite{HVA13}, Elsenhans has implemented $p$-adic cohomology point counting methods for K3 surfaces of degree $2$ in {\tt Magma}~\cite{MAGMA}.  In particular, this allows us to do point counts in characteristic $13$, which are infeasible with more naive algorithms.  This is particularly noteworthy, because the K3 surface $Y$ has bad reduction at $3$ and $5$, so point counts in these characteristics need not yield information on the geometric Picard rank of $Y$.

\begin{prop}\label{prop: K3hasPicRk1} The surface $Y$ has geometric Picard rank 1.  Consequently, $\Br_1(Y)/\Br_0(Y)$ is trivial.
\end{prop}

\begin{proof}
	Briefly, if $p$ is a prime of good reduction for $Y$, then $\Pic(\Ybar) \hookrightarrow \Pic(\Ybar_p)$~(see, e.g.,~\cite{vLHeron}*{Proposition~6.2}).  We show that $\Pic(\Ybar_{13})$ has rank $2$, and is generated by the components of the pullback of a tritangent line to the branch sextic of the double cover $\Ybar_{13} \to \PP^2_{\Fbar_{13}}$.  Meanwhile, the branch curve of the double cover $\Ybar_{7} \to \PP^2_{\Fbar_{7}}$ does not admit tritangent lines. An application of~\cite{HVA13}*{Proposition 5.3} allows us to conclude that $\Pic(\Ybar) \isom \Z$.

	To check that the branch sextic in $\PP^2_{\Fbar_{13}}$ admits tritangent lines while the branch sextic in $\PP^2_{\Fbar_{7}}$ does not, we use \cite{EJ08}*{Algorithm~8}. We find that $\Ybar_{13}$ is given by
	
	\begin{align*}
		w^2 & = 10(y + 11z)^2(y^2 + 5yz + 8z^2)^2 \\
		& + (x + 5y + z)(x^5 + 3x^4y + x^4z + 11x^3y^2 + 8x^3yz + 9x^3z^2 + 2x^2y^3 \\
		& + 10x^2y^2z + 9x^2yz^2 + x^2z^3 + 8xy^4 + 7xy^3z + 10xy^2z^2 \\
		& + 10xyz^3 + 5xz^4 + 11y^5 + y^4z + 10y^2z^3 + 10yz^4 + z^5),
	\end{align*}

\noindent from which one can see that the line $x + 5y + z = 0$ is tritangent to the branch curve of the double cover. The components of the pullback of this line to $\Ybar_{13}$ generate a rank $2$ sublattice of $\Pic(\Ybar_{13})$.  Let $F(t)$ denote the characteristic polynomial of the Frobenius operator on $\HH^2_{\et}\left(\Ybar_{13},\Q_\ell\right)$ for some prime $\ell \neq 13$.  Set $F_{13}(t) = 13^{-22}F(13t)$; the function {\tt WeilPolynomialOfDegree2K3Surface} in {\tt Magma} computes
\begin{align*}
F_{13}(t) & = \frac{1}{13}(t - 1)^2(13t^{20} + 8t^{19} - 6t^{17} - 3t^{16} + 3t^{15} + 3t^{14} + 8t^{13} \\ 
&- 2t^{11} - 3t^{10} - 2t^9 + 8t^7 + 3t^6 + 3t^5 - 3t^4 - 6t^3 + 8t + 13).
\end{align*}
As the number of roots of $F_{13}(t)$ that are roots of unity gives an upper bound for $\rk\Pic(\Ybar_{13})$ \cite{vL07}*{Corollary~2.3}, we conclude that $\Pic(\Ybar_{13}) \isom \Z^2$ (note that the roots of the degree $20$ factor in $F_{13}(t)$ are not integral, so they cannot be roots of unity). A computation shows that $Y$ has no line tritangent to the branch curve when we reduce modulo $p = 7$. This concludes the proof that $\Pic(\Ybar) \isom \Z$.

For the last claim, recall that the group $\Br_1(Y)/\Br_0(Y)$ is isomorphic to the Galois cohomology group $\HH^1(\Q,\Pic(\Ybar))$.  Since $\Pic(\Ybar) \isom \Z$, we know that $\Pic(\Ybar)$ is a trivial Galois module.  The first cohomology group of a free module with a trivial action by a profinite group being trivial, we deduce that $\Br_1(Y)/\Br_0(Y) = 0$, as desired.
\end{proof}

\subsection{Local Invariants}

In this section we compute the local invariants of the algebra $\alpha$ for our particular surface $Y$. 

\begin{prop}\label{prop: LocInvtsOfEx}
	Let $p \le \infty$ be a prime number. For any $P \in Y(\Q_p)$ we have
	\[
		\inv_p \left( \alpha(P) \right) = 
		\begin{cases}
			0, & \textup{ if } p \ne 3 \\
			\frac{1}{3} \textup{ or } \frac{2}{3}, & \textup{ if } p = 3
		\end{cases} 
	\]
\end{prop}

\begin{proof}
	Whenever $p$ is a finite prime of good reduction for $Y$, we have $\inv_p (\alpha(P)) = 0$ for all $P$ by Lemma \ref{lem: PlacesOfRamification}.

	The fourfold $X$ is constructed so that it is $3$-adically insoluble.  Since $X$ is birational to $X'$, the Lang--Nishimura lemma implies that $X'(\Q_3) = \emptyset$, and thus by Lemma \ref{lem: InsolubleCubic} applied to $p=3$, we have $\inv_3 \left( \alpha(P) \right) \ne 0$ for every point $P \in Y(\Q_3)$. 

	At every prime $p \ne 3$ of bad reduction of $Y$, the singular locus consists of $r < 8$ ordinary double points. A straightforward computer calculation verifies this claim. Hence, Corollary \ref{cor: ConstInvtsMildBadRed} 
	implies that for all such $p$, the invariant $\inv_p(\alpha(P))$ is trivial.
\end{proof}

\begin{remark} Although we are able to conclude the $3$-adic invariant is non-trivial, our approach does not allow us to distinguish between the possibilities of $\frac{1}{3}$ and $\frac{2}{3} \in \frac{1}{3}\Z/\Z$. 
\end{remark}

\begin{proof}[Proof of Theorem \ref{thm: MainThm}]
	Let $Y/\Q$ be the surface considered throughout this section. By Proposition~\ref{prop: K3hasPicRk1}, we know that $\Br_1(Y) = \Br_0(Y)$, and hence the class $\alpha \in \Br(Y)[3]$ is transcendental as long as it is not constant.

	In \S\ref{subsec: LocalPoints}, we verified $Y(\A_{\Q}) \ne \emptyset$. On the other hand, by Proposition \ref{prop: LocInvtsOfEx}, $Y(\A_{\Q})^{\alpha} = \emptyset$, thereby demonstrating \emph{a posteriori} that $\alpha$ is nonconstant.
\end{proof}

\section{Systematic construction of examples}\label{sec:ComputationalTricks}

In this section we outline a sequence of steps, with some commentary, explaining how to organize the computation of examples like the one in \S\ref{sec:HPCounterexample}.  We highlight computational difficulties that must be overcome to make the whole construction feasible.

\subsection*{Step 1: Insoluble plane cubics give rise to insoluble cubic fourfolds} As in the beginning of \S\ref{sec:HPCounterexample}, we begin with a choice of quadrics $Q_1$, $Q_2$, $Q_3$ in $\Z[x_0,\dots,x_5]$ containing the two disjoint planes $\Pi_1 := \{x_0 = x_1 = x_2 = 0\}$ and $\Pi_2 := \{x_3 = x_4 = x_5 = 0\}$ in $\PP^5$. Generators for the ideal of the associated elliptic ruled surface $T$ are obtained by saturating out the product ideal $I(\Pi_1)I(\Pi_2)$ from $\langle Q_1, Q_2, Q_3 \rangle$, and constructing a minimal basis for the resulting ideal. The result is $I(T) = \langle Q_1,Q_2,Q_3,C_1,C_2\rangle$, where $C_1$ and $C_2$ are homogeneous cubic polynomials in $\Z[x_0,\dots,x_5]$.

We carefully choose the quadrics so that $C_1$ and $C_2$ satisfy two key properties: 
\begin{enumerate}
	\item $C_1$, $C_2$ each lie in a three-variable subring of $\Q[x_0,\dots,x_5]$, with disjoint variable sets;
	\item when considering the zero locus of $C_1$ and $C_2$ as plane cubic curves, each curve has no $(\Z/3\Z)$-points. 
\end{enumerate}
We note that a cubic in at least 4 variables always has $(\Z/3\Z)$-points by the Chevalley-Warning theorem, so the restriction on number of variables in $(1)$ above is necessary for $(2)$ to be possible. Moreover, as $\approx 99.259 \%$ of plane cubic curves over $\Q_3$ have $\Q_3$-rational points (\cite{BCF16}*{Theorem 1}), we must sample a large quantity of quadrics to achieve cubics with the desired properties. To construct a $\Q_3$-insoluble cubic fourfold $X$ containing $T$, we proceed as in \S\ref{sec:HPCounterexample}. Taking $\{3C_1 + C_2 = 0\}$ yields a cubic fourfold which is insoluble in $\Z/9\Z$ (and hence over $\Q_3$), but need not be smooth, so we modify by a suitable linear combination of cubics in $\langle Q_1, Q_2, Q_3 \rangle$. This yields a cubic fourfold $X$ that contains a sextic elliptic surface $T$, such that $X(\Q_3) = \emptyset$.

\subsection*{Step 2: Equation for the K3 sextic} To determine the degree $12$ discriminant locus $\{f(x,y,z) = 0\}$ of the morphism $\pi \colon X' := \Bl_T(X) \to \PP^{2}_{[x:y:z]} := |\calI_T(2)|$, we employ an interpolation approach, as the natural Gr\"obner basis elimination computation over characteristic zero is too expensive. In particular, we sample integral coprime pairs $(y,z) = (y_0, z_0)$ of small height; each choice yields a degree 12 polynomial in $x$, which is the specialization $f(x,y_0,z_0)$ of $f(x,y,z)$. Writing $f(x,y,z) = \sum_{i,j,k} a_{ijk}x^i y^j z^k$, each specialization gives $13$ linear relations among the indeterminates $a_{ijk}$, one for each coefficient of $x^i$ in $f(x,y_0,z_0)$. With enough specializations, one may reconstruct $f(x,y,z)$, using Gaussian elimination. The sextic corresponding the K3 surface $Y$ is then determined as the irreducible component of the discriminant locus which is nonsingular (i.e., it is $B_I$ in the notation of~\cite{AHTVA}). 

\subsection*{Step 3: Existence of a smooth model for $Y$ at $p=2$} In order for the K3 surface $Y$ to have good reduction at 2, it is necessary to have a model for $Y$ which over $\F_2$ reduces to an equation of the form $w^2 + wg_1 + g_2 = 0$. If $Y/\Q$ is defined initially by an equation of the form $w^2 = f$ such that over $\F_2$ the equation becomes $\overline{w}^2 = g_1^2$, then we make the rational change of variables $w \mapsto 2w + g_1$. This gives a model for $Y$ of the form $4w^2 + 4wg_1 + g_1^2 = f$. Thus, a sufficient condition to produce a model of the desired form is to additionally require that $f - g_1^2 \equiv 0 \bmod{4}$, as this guarantees that all coefficients are divisible by 4. Dividing out by this factor we obtain a new model for $Y$, whose reduction modulo 2 has the shape described above. Finally, we proceed as in \S\ref{subsec: PrimesOfBadRed} to check for smoothness.

\subsection*{Step 4: Local points at small primes of good reduction} For primes $p \le 22$ and $p = \infty$, test for $\Q_p$ points of $Y/\Q \colon w^2 = f(x_,y,z)$, by evaluating at integers with small absolute value (typically 0 or 1) for $x,y$ and $z$, and determining whether $f$ is a square in $\Q_p$. If this test fails, then it is plausible that $Y$ has no local points; start over.

\subsection*{Step 5: Twist} In Step 2, we produce the smooth plane sextic $f(x,y,z)$ over $\Q$, which gives rise to a K3 surface with defining equation $\delta w^2 = f(x,y,z)$. As discussed in \S\ref{sec:HPCounterexample}, to determine $\delta$ it suffices to consider any smooth fiber $S_{[x_0, y_0, z_0]} = \pi^{-1}([x_0, y_0, z_0])$ of $\pi \colon X' \to \PP^2$. In particular, $\delta = 1$ (up to squares), if and only if each of the skew triples of $(-1)$-curves is defined over the extension $K_0 := \Q(\sqrt{f(x_0, y_0, z_0)})$.

We assume now that $S := S_{[x_0, y_0, z_0]}$ is anticanonically embedded in $\PP^6_{[X_0, \dots, X_6]}$, cut out by the vanishing of 9 quadrics. Define the degree 0 graded $\Q$-homomorphism
\begin{eqnarray*} 
	\phi \colon \Q[X_0, \dots, X_6] & \to & \Q[s,t] \\
	X_i & \mapsto & a_i s + b_it,
\end{eqnarray*}
and consider $a_0, \dots, a_6, b_0, \dots, b_6$ as indeterminates, subject to the restriction that the rank of the matrix
\begin{equation} \label{eqn: matrix}
	\begin{pmatrix} 
		a_0 & a_1 & a_2 & a_3 & a_4 & a_5 & a_6 \\
		b_0 & b_1 & b_2 & b_3 & b_4 & b_5 & b_6
	\end{pmatrix} 
\end{equation}
is maximal. Substitute the expressions $X_i = a_i s + b_i t$ into each quadric, and consider each result as the zero polynomial in the variables $s,t$ with coefficients in $R := \Q[a_0, \dots, a_6, b_0 \dots, b_6]$. That is, apply $\phi$ to the ideal $I_S$ of quadrics defining $S$ to obtain an ideal in $R$. Let $I_{\vec{a}, \vec{b}}$ be the homogeneous ideal obtained by saturating this ideal by the ideal of $2 \times 2$ minors of ~\eqref{eqn: matrix}. Let $I_S^e$ and $I_{\vec{a}, \vec{b}}^e$ be the respective extensions of $I_S$ and $I_{\vec{a}, \vec{b}}$ to the ring \[ \Q[X_0, \dots, X_6, a_0, \dots, a_6, b_0 \dots, b_6,s,t]. \]
Eliminating all variables except $X_0, \dots, X_6$ from the ideal 
\[ I_S^e + I_{\vec{a}, \vec{b}}^e + \langle X_0 - (a_0 s + b_0 t), \dots, X_6 - (a_6 s + b_6 t)\rangle \]
gives an ideal whose contraction to $\Q[X_0, \dots, X_6]$ is a homogeneous ideal cutting out precisely the scheme of $(-1)$-curves in $S$. After base extension to $K_0$, compute the primary decomposition of this ideal. If the corresponding irreducible components of the scheme of $(-1)$-curves are not of degree 3, start over; $\delta \ne 1$, and we cannot determine the twist using this approach.

\subsection*{Step 6: Geometric Picard rank 1} Let $C$ be the smooth plane curve in $\PP^2$ defined by the vanishing of the K3 sextic $f(x,y,z)$. Compute \[ S := \{ p \mid 5 \le p \le 100 \textup{ a prime of good reduction for } C \}. \]
Use \cite{EJ08}*{Algorithm 8} to compute the subset $S_\textup{tri} \subset S$ consisting of those primes $p$ for which there exists a line tritangent to $C$ over $\F_p$. If either $S_\textup{tri}$ or $S \setminus S_\textup{tri}$ are empty, start over. Otherwise, starting with the smallest prime $p$ in $S_\textup{tri}$, compute the characteristic polynomial $F(t)$ of Frobenius on $H^2_{\et}(\Ybar_p, \Q_\ell)$ for some prime $\ell \ne p$, using the {\tt Magma} function {\tt WeilPolynomialOfDegree2K3Surface}. The number of roots of unity of $F_p(t) := p^{-22} F(pt)$ gives an upper bound on $\rk \Pic \Ybar_p$. If this exceeds 2 for all primes in $S_\textup{tri}$, start over; we cannot verify geometric Picard rank 1 for this surface.

\subsection*{Step 7: Primes of bad reduction} Proceeding as in \S\ref{subsec: PrimesOfBadRed}, compute the integer $m$ whose prime factors comprise the primes of bad reduction for the sextic $f(x,y,z)$ (and hence for the K3 surface $Y$). Monomial ordering plays a crucial role in this step; grevlex is used for efficient Gr\"obner basis calculations. The integer $m$ contained in the Gr\"obner basis with respect to grevlex ordering is necessarily contained in the ideal generated by a Gr\"obner basis $G'$ in lexicographic monomial ordering, which can be used for elimination. Thus $m \in G' \cap \Z = (m')$, where $m'$ is the integer whose prime factors are the primes of bad reduction. Hence $m' \mid m$, and thus the primes of bad reduction divide $m$, which suffices for our purposes.

\subsection*{Step 8: Computations at primes of bad reduction} At the places of bad reduction, check for local points, as in Step 4. Determine the (geometric) singular locus. If at any prime in question the locus does not consist of $r < 8$ ordinary double points, then start over.

\bibliographystyle{plain}

\end{document}